\documentclass[11pt]{amsart}

\usepackage{geometry}
\usepackage{empheq}        
\usepackage{amssymb}       
\usepackage{amsthm}        
\usepackage{esint}         
\usepackage{dsfont}        
\usepackage{mathrsfs}      

\usepackage{xcolor}        
\usepackage{graphicx}      
\usepackage{tikz}          
\usetikzlibrary{cd}        

\usepackage{array}
\usepackage{enumitem}
\usepackage{tabularx, booktabs}

\usepackage[abbrev]{amsrefs}

\usepackage[colorlinks=true, allcolors=blue]{hyperref}
\hypersetup{
    colorlinks = true,
	allcolors = blue
}

\theoremstyle{plain}
\newtheorem{thm}{Theorem}[section]
\newtheorem{cor}[thm]{Corollary}

\newtheorem{lem}[thm]{Lemma}
\newtheorem{prop}[thm]{Proposition}

\theoremstyle{definition}
\newtheorem{defn}[thm]{Definition}
\newtheorem{ex}[thm]{Example}
\newtheorem{question}[thm]{Question}

\newtheorem{rem}[thm]{Remark}

\numberwithin{equation}{section}

\def\sband{A_{\mathrm{sph}}}

\DeclareMathOperator{\supp}{supp}
\DeclareMathOperator{\Real}{Re}
\DeclareMathOperator{\Hess}{Hess}

\begin{document}
\title{Payne--Philippin's overdetermined problems on compact surfaces}

\author{Hang Chen}
\address[Hang Chen]{School of Mathematics and Statistics, Northwestern Polytechnical University, Xi' an 710129, P. R. China \\ email: chenhang86@nwpu.edu.cn}
\thanks{Chen is supported by NSFC Grant No.~12571054 and Natural Science Foundation of Shaanxi Province Grant No.~2024JC-YBMS-011}
\author{Bohan Wu}
\address[Bohan Wu]{School of Mathematics and Statistics, Northwestern Polytechnical University, Xi' an 710129, P. R. China \\ email: mod@mail.nwpu.edu.cn}

\begin{abstract}
	We investigate the overdetermined problem given by
	\begin{equation*}
	\Delta u=0 \text{ in } \Omega,\quad  \frac{\partial u}{\partial\nu} =\sigma_1 u  \text{ on } \partial \Omega, \quad |\nabla u|=\text{constant on } \partial \Omega,
	\end{equation*}
	where $\Omega$ is a connected compact Riemannian surface with smooth boundary $\partial \Omega$,
	and $\sigma_1$ is the first nonzero Steklov eigenvalue of $\Omega$.
	We prove that this overdetermined problem admits a nontrivial solution if and only if $\Omega$ is $\sigma$-homothetic to either the flat unit disk or a flat cylinder $[-T,T]\times S^1$ for some $T\ge T_1$.
	This gives a complete answer to the question raised by Payne and Philippin in [Z.~Angew.~Math.~Phys. \textbf{42}(6), 864--873, 1991] for $\sigma=\sigma_1$ and arbitrary surfaces.
	In particular, we completely characterize compact domains in 2-dimensional space forms for which the overdetermined problem is solvable.
\end{abstract}


\maketitle

\section{Introduction}\label{sect-intro}
Finding solutions to partial differential equations under prescribed boundary conditions is a fundamental and important topic in analysis.
Overdetermined problems, which impose more boundary conditions than typically required for well-posedness, often lead to symmetry conclusions.
A celebrated result is Serrin's theorem \cite{Ser71},
which (in a special case) states that, for a domain $\Omega\subset \mathbb{R}^n$, the solvability of the system
\begin{equation*}
	\left\{\begin{aligned}
		\Delta u&=-1 && \text{ in } \Omega,\\
		u&=0 && \text{ on } \partial\Omega,\\
		\partial_{\nu}u&=\text{constant} && \text{ on } \partial\Omega
	\end{aligned}\right.
\end{equation*}
implies that $\Omega$ is a ball and $u$ is radially symmetric.
Since Serrin's work, numerous extensions have been obtained.
For instance, Kumaresan and Prajapat \cite{KP98} generalized the result to the hemisphere and the hyperbolic space;
Xia and his collaborators \cites{GX19,JLXZ24} studied so-called partially overdetermined problem for domains of manifolds with boundary;
Gao, Ma and Yang \cite{GMY23} considered overdetermined problems for fully nonlinear equations in space forms.

This paper is devoted to overdetermined Steklov eigenvalue problems.
Let $(M^n, g)$ be an $n$-dimensional complete Riemannian manifold, $n\ge 2$.
For a bounded domain $\Omega\subset M$ with smooth boundary $\partial \Omega$, the Steklov problem is to find real numbers $\sigma$ such that there exists a nonzero function $u$ on $\Omega$ satisfying
	\begin{empheq}[left=\empheqlbrace]{alignat=3}
		\Delta u&=0  & & \text{ in } \Omega,\label{eq-s1}\\
	\partial_{\nu} u &=\sigma u  \quad & & \text{ on } \partial \Omega,\label{eq-s2}
	\end{empheq}
where $\nu$ is the outward unit normal to $\partial \Omega$.
We call such a $\sigma$ a Steklov eigenvalue,
and $u$ the corresponding Steklov eigenfunction.
It is well known that Steklov eigenvalues are discrete and satisfy
\begin{equation*}
	0=\sigma_0<\sigma_1\le \sigma_2\le \cdots \to +\infty.
\end{equation*}

Over the past decades, the Steklov spectrum has been intensively studied.
Various results are obtained, such as estimates of the bounds, comparison theorems and isoperimetric inequalities; see \cites{FS16, Xio18, GLW25} for instance.
We refer the reader to \cite{CGGS24} for a recent survey which includes this topic.
However, to the authors’ knowledge, few results address overdetermined Steklov eigenvalue problems.
Payne and Philippin \cite{PP91} first imposed the following additional condition
	\begin{equation}
	|\nabla u| =c \text{ on } \partial \Omega\label{eq-s3}
	\end{equation}
and raised the following questions:
\begin{question}\label{ques-1}
	Is there a domain $\Omega$ for which the overdetermined problem has a solution? If so for which eigenfunctions is it achieved? Is the domain $\Omega$ unique and if so what is this domain?
\end{question}
In the same paper, they studied the above overdetermined problem for planar domains and proved the following theorem.
\begin{thm}[\cite{PP91}*{Theorem 1}]\label{thm:PP1}
	Let $\Omega\subset \mathbb{R}^2$ be a simply connected domain with smooth boundary $\partial \Omega$.
	Then the overdetermined problem \eqref{eq-s1}, \eqref{eq-s2} and \eqref{eq-s3} with $\sigma=\sigma_1$ is solvable
	if and only if $\Omega$ is a round disk.
\end{thm}
We remark that the constant $c$ in \eqref{eq-s3} must be positive when $u$ is an eigenfunction corresponding to a nonzero Steklov eigenvalue.
Otherwise, $u\equiv 0$ on $\partial\Omega$ by \eqref{eq-s2}, and then $u\equiv 0$ in $\Omega$ by \eqref{eq-s1}, a contradiction.

A natural question is whether Theorem~\ref{thm:PP1} can be generalized.
In \cite{PP91}, Payne and Philippin conjectured that Theorem~\ref{thm:PP1} holds for $\sigma_k (k\ge 2)$ and $\Omega\subset\mathbb{R}^n (n\ge 3)$.
However, this conjecture turned out to be false.
For planar domains, Alessandrini and Magnanini \cite{AM94} established a relation between a solution $u$ and the associated holomorphic function $F$ (called ``\emph{complex potential}'').
By analyzing the zeros of $F$ in $\Omega$, they
constructed various non-symmetric \emph{solvable domains} (cf.~Definition~\ref{def-sol-dom}).
For $\Omega\subset\mathbb{R}^n (n\ge 3)$,
Alessandrini and Magnanini \cite{AM96} connected solutions to isoparametric functions and isoparametric hypersurfaces, and proved an analogue of Theorem~\ref{thm:PP1} under additional conditions.

Very recently, Lee and Seo \cite{LS25} studied the surface case.
Observing the condition \eqref{eq-s3} is equivalent to
\begin{equation*}
	\Hess u(\nabla u) =\psi \nu  \text{ for some $\psi\in C^\infty(\partial\Omega) $ on } \partial \Omega,
\end{equation*}
they established
\begin{thm}[\cite{LS25}*{Theorem 2.3 and Corollary 2.4}]\label{thm:LS}
	Let $\Omega$ be a connected compact surface with $C^2$ boundary.
	Assume that the Gaussian curvature $K_\Omega>0$ and $\psi\le 0$.
	If the overdetermined problem \eqref{eq-s1}, \eqref{eq-s2} and \eqref{eq-s3} with $\sigma=\sigma_1$ admits a nontrivial solution,
	then $\Omega$ is flat and $\partial\Omega$ is the
	disjoint union of geodesics and geodesic circles of radius $1/\sigma_1$.
	In particular, if $\Omega$ is simply connected, then $\Omega$ is a flat disk.
\end{thm}

We study above overdetermined problem for arbitrary compact Riemannian surfaces, dropping both the simply connected condition in Theorem~\ref{thm:PP1} and the restrictions on $K_\Omega$ and  $\psi$ in Theorem~\ref{thm:LS}.
Due to the examples constructed by Alessandrini and Magnanini \cite{AM94}*{Theorem 1.2},
a complete classification of solvable surfaces seems elusive;
therefore, we focus on the case $\sigma=\sigma_1$, and provide a further discussion for general $\sigma$.

For convenience, throughout this paper, all the surfaces (or domains) and there boundaries are assumed to be smooth.
We introduce some definitions and notations.

\begin{defn}\label{def-sol-dom}
	A surface (or a domain) $\Omega$ is called \emph{solvable} for an overdetermined problem if it admits a nontrivial solution of this overdetermined problem.
	Otherwise, we say that $\Omega$ is \emph{unsolvable} for this overdetermined problem.
\end{defn}

For convenience, we refer to the overdetermined Steklov problem \eqref{eq-s1}--\eqref{eq-s3} with $\sigma=\sigma_1$ as (PP-I),
i.e.,
\begin{equation}\label{PP-I}
	\left\{
	\begin{aligned}
		\Delta u&=0 && \text{ in } \Omega,\\
		\partial_{\nu}u&=\sigma_1 u && \text{ on } \partial\Omega,\\
		|\nabla u|&=c \text{ a constant} && \text{ on } \partial\Omega.
	\end{aligned}\right.\tag{PP-I}
\end{equation}

Under conditions on the Steklov spectrum,
as pointed out in \cite{FS16},
one can only hope to classify surfaces up to the  equivalence relation called ``$\sigma$-homothetic'' (cf.~Lemma~\ref{lem-conf} for details). This terminology was introduced by Fraser and Schoen.
\begin{defn}[cf.~\cite{FS16}*{Definition 2.1}]
	For two surfaces $(\Omega_1,g_1)$ and $(\Omega_2,g_2)$,  we say that $\Omega_1$ and $\Omega_2$ are \emph{$\sigma$-homothetic} if there is a conformal diffeomorphism $\Phi \colon \Omega_1\to \Omega_2$ such that the pullback metric $\Phi^{\ast}g_2 = \rho^2 g_1$ with $\rho|_{\partial \Omega_1} = c$ for a constant $c$. Such $\Phi$ is called a \emph{$\sigma$-homothety} (from $\Omega_1$ to $\Omega_2$).
\end{defn}

\noindent
\textbf{Notations:}
\begin{itemize}
	\item $\mathbb{M}^2_K$: the 2-dimensional simply connected space form of constant Gaussian curvature $K$, i.e., $\mathbb{M}^2_K=\mathbb{S}^2$, $\mathbb{R}^2$ and $\mathbb{H}^2$ when $K=1,0$ and $-1$, respectively.
	\item $\mathbb{D}$: the unit disk of $\mathbb{R}^2$.
	\item $\mathbb{A}_T=[-T,T]\times S^1$: a cylinder with  the flat metric $\mathrm{d}t^2+\mathrm{d}\theta ^2$, where $L>0$ and $S^1$ is the unit circle.
	\item $T_1(\approx 1.2)$: the unique positive solution of the equation $1/x=\tanh x$.
	\item $\sband(R)$: the spherical zone in $ \mathbb{S}^2$ which is bounded by two parallel latitude circles of the same radius $\cos R$, which is symmetric with respectively to the equator,
	i.e.,
	\begin{equation*}
		\sband(R)=\{(\cos r \cos\theta, \cos r \sin\theta, \sin r)\mid -R\le r\le R, \theta \in S^1\},
	\end{equation*}
	where $0<R< \pi/2$.
	\item $R_1 = 2 \arctan \tanh (T_1/2)\approx 0.99 \approx 56.7^{\circ}$.
	\item $(\Omega_1,g_1)\sim (\Omega_2,g_2)$: $\Omega_1$ and $\Omega_2$ are $\sigma$-homothetic.
\end{itemize}

We now state our main result, which provides a complete answer to Question \ref{ques-1} for $\sigma=\sigma_1$ and any compact surfaces with smooth metrics.

\begin{thm}\label{thm-PP-I}
	Let $(\Omega,g_\Omega)$ be a connected compact surface with boundary $\partial \Omega$.
	Then $\Omega$ is solvable for \eqref{PP-I} if and only if $\Omega$ is $\sigma$-homothetic to either $ \mathbb{D}$ or $\mathbb{A}_T$ for some $T\ge T_1$.

	Furthermore, the solution set $\mathcal{S}$ of \eqref{PP-I} is a linear space and can be characterized as follows:

	(1) For $(\Omega,g_\Omega)\sim \mathbb{D}$, let $\Phi\colon\Omega \to \mathbb{D}$ be a $\sigma$-homothety.
	Then $\mathcal{S}$ coincides with the first Steklov eigenspace of $\Omega$, spanned by $\Phi^{\ast}(r\cos \theta)$ and $\Phi^{\ast}(r\sin \theta)$.

	(2) For $(\Omega,g_\Omega)\sim \mathbb{A}_T, T\ge T_1$, let $\Phi \colon \Omega \to \mathbb{A}_T$ be a $\sigma$-homothety.
	Then $\mathcal{S}$ is the subspace of the first Steklov eigenspace of $\Omega$ spanned by $\Phi^{\ast}t$.
\end{thm}

\begin{rem}
	(1) In this paper, we always treat the constant $c$ in \eqref{PP-I} as non-fixed; that is, we only require that $|\nabla u|$ be some constant on the boundary. If $c$ is prescribed, the solution is determined up to scaling of any nontrivial $u\in\mathcal{S}$.

	(2) The second part of Theorem~\ref{thm-PP-I} follows directly from Example~\ref{ex-disk} and Example~\ref{ex-annulus}.
	Furthermore, we see that, if $\mathcal{S}$ coincides with the first Steklov eigenspace of $\Omega$,
	then the critical cylinder $\mathbb{A}_{T_1}$ must be excluded.

	(3) It is worth noting that the critical cylinder $\mathbb{A}_{T_1}$ is $\sigma$-homothetic to
	a so-called ``critical catenoid'' in \cite{FS11}, which is of geometric interest since it is a free boundary solution in the ball,
	and it is the unique maximizer (up to $\sigma$-homothety) of $\sigma_1(\Omega)L(\partial\Omega)$ among all smooth annuli.
	The readers may refer to \cites{FS11, FS16} for details.
	We may call $\mathbb{A}_{T}$ a supercritical cylinder when $T\ge T_1$.
\end{rem}

We will separate Theorem~\ref{thm-PP-I} into two parts: the simply connected case Theorem~\ref{thm-sim-conn} and the multiply connected case Theorem~\ref{thm-multi-conn}.
Furthermore,
applying this theorem to domains of $\mathbb{M}_K^2$, we prove the following results (cf.~Corollaries~\ref{cor-sim-conn}, \ref{cor-multi-conn-euc-hyp} and \ref{cor-multi-conn-sph}).
\begin{thm}\label{thm-euc-hyp}
	A connected compact domain $\Omega\subset\mathbb{R}^2$ or $\Omega\subset\mathbb{H}^2$ with boundary $\partial \Omega$
	is solvable for \eqref{PP-I} if and only if $\Omega$ is a geodesic disk.
\end{thm}
\begin{thm}
	A connected domain $\Omega\subset\mathbb{S}^2$ with boundary $\partial \Omega$
	is solvable for \eqref{PP-I} if and only if $\Omega$ is either a spherical cap or $\sigma$-homothetic to
	$\sband(R)$ for some $R\in [R_1,\pi/2)$.
	In the latter case, $\Omega$ is symmetric with respect to a certain great circle of $\mathbb{S}^2$.
\end{thm}
\begin{cor}
	A connected domain $\Omega\subset\mathbb{S}^2_{+}$ with boundary $\partial \Omega$
	is solvable for \eqref{PP-I} if and only if $\Omega$ is a spherical cap.
\end{cor}

\smallskip

The paper is organized as follows.

In Sect.~\ref{Sect:Pre}, we recall some fundamental computations under a special local coordinate system in Riemannian geometry and basic properties of Steklov eigenfunctions.

In Sect.~\ref{Sect:simply}, we solve the case of the simply connected surfaces,
where Proposition~\ref{prop-2} plays a crucial role.
Regrading an auxiliary function, finer and novel calculations and analyses are done.
This ultimately allows us to employ the Gauss--Bonnet formula and the Weinstock inequality and prove Theorem~\ref{thm-sim-conn}.

Theorem~\ref{thm:PP1} can be derived by Theorem~\ref{thm-sim-conn} (cf.~Corollary~\ref{cor-sim-conn}).
This means that, even for the simply connected domains,
we give an alternative proof of Theorem~\ref{thm:PP1}, which is more geometric in flavor, whereas the original one is more analytic.

Next, we deal with the case of the multiply connected surfaces in Sect.~\ref{Sect:multiply}.
We firstly check the solvability of an annulus (Example~\ref{ex-annulus}), where the key Proposition~\ref{prop-2} fails.
We seek a modified version Proposition~\ref{prop-3} instead of Proposition~\ref{prop-2}, then the Gauss--Bonnet formula can be applied again and help us determine the genus and the number of the boundary components.
Furthermore, we show that boundary components are of equal length, which is key to prove the restriction of the  conformal transformation to the boundary is an isometry.

As an application, we investigate which multiply connected domains in a space forms are solvable for \eqref{PP-I}.
The results are obtained by discussing the existence and uniqueness of solutions of the second-order elliptic equation satisfied by the conformal factor.
Corollary~\ref{cor-multi-conn-euc-hyp} (cf.~Theorem~\ref{thm-euc-hyp}) implies that the simply connected condition in Theorem~\ref{thm:PP1} is not necessary.

In Sect.~\ref{Sect:general-case},
we provide further discussions on several extensions.
First, by relaxing the condition that the boundary gradient magnitude be a global constant to being locally constant,
we establish a classification for this weaker overdetermined problem.
Second, for higher eigenvalues,
under the additional assumption that the overdetermined problem admits a solution without critical points,
we prove a classification result.
Finally, we derive a necessary condition for the overdetermined problem to be solvable in the general case.

In the last Sect.~\ref{Sect:proof-thmC},
we investigate another type of overdetermined condition, namely, that the mean value of each nonconstant Steklov eigenfunction over the domain vanishes.
We prove that, for any compact connected domain in the hemisphere or hyperbolic space of arbitrary dimension, this overdetermined condition holds if and only if the domain is a geodesic ball (Theorem~\ref{thm-C}),
where Serrin-type theorems for $\mathbb{S}^n$ and $\mathbb{H}^n$ is applied.
\section{Preliminaries}\label{Sect:Pre}

\subsection{Normal exponential maps}\label{Section:nep}
Let $(\Omega, g_\Omega)$  be a simply connected compact surface with smooth boundary $\partial\Omega$ of length $L$.
We parametrize $\partial \Omega$ by the arc-length as  $\gamma \colon [0,L]\to \partial \Omega, \gamma(0)=\gamma(L), \gamma'(0)=\gamma'(L), |\gamma'|\equiv 1$.
Points in $\Omega$ near the boundary can be parametrized by $[0,\epsilon)\times [0,L] \ni(x,y)\mapsto \exp_{\gamma(y)}(x\hat{\nu})$ for small $\epsilon>0$, where $\hat{\nu}=-\nu$ is the inward unit normal along the boundary.
Then the metric on $\Omega$ (near the boundary) can be written as
	\begin{equation*}
		g_{\Omega}=\mathrm{d}x^{2}+\lambda^{2}(x,y)\mathrm{d}y^{2},
	\end{equation*}
	where $\lambda(x,y)$ is a positive function satisfying $\lambda(0,y)\equiv 1$.

If we write $g_\Omega=g_{ij}\mathrm{d}v^i\mathrm{d}v^j$ with $(v^1,v^2)=(x,y)$, then the Christoffel symbols are given by
	\begin{gather*}
    \Gamma_{11}^{1}=\Gamma_{12}^{1}=\Gamma_{21}^{1}=\Gamma_{11}^{2}=0,\\
	\Gamma_{22}^{1}=-\lambda\lambda_{x},\quad\Gamma_{22}^{2}=\frac{\lambda_{y}}{\lambda},\quad\Gamma_{12}^{2}=\Gamma_{21}^{2}=\frac{\lambda_{x}}{\lambda}.
	\end{gather*}

	For any smooth function $\phi\in C^\infty(\Omega)$, under the local coordinates above, it is not hard to obtain
	\begin{gather}
		\nabla \phi =\phi_x\partial_x+\frac{1}{\lambda^{2}}\phi_y\partial_y,\label{eq2.1}\\
		|\nabla \phi|^2=\phi_x^2+\frac{1}{\lambda^{2}}\phi_y^2,\label{eq2.2}\\
		\Delta \phi =\phi_{xx}+\frac{1}{\lambda^{2}}\phi_{yy}+\frac{\lambda_x}{\lambda}\phi_x-\frac{\lambda_y}{\lambda^3}\phi_y.\label{eq2.3}
	\end{gather}

A direct computation shows that $\nabla_{\gamma'(y)}\gamma'(y)=(-\lambda_x(0,y))\hat{\nu}$.
Hence,
	\begin{equation}\label{eq-geo-cur}
		\kappa=-\lambda_x(0,y)
	\end{equation}
is the geodesic curvature $\kappa$ of $\partial\Omega$ (with respect to the inward normal $\hat{\nu}$).

\subsection{Basic properties of Steklov eigenfunctions}
Although the following result is well-known, we give a quick proof.
\begin{lem}\label{lem-2.2}
Assume that $\partial\Omega$ has only one connected component.
Then no Steklov eigenfunction of $\Omega$ corresponding to the eigenvalue $\sigma_k (k\ge 1)$ is constant along $\partial\Omega$.
\end{lem}

\begin{proof}
	For any Steklov eigenfunction $u$ corresponding to  $\sigma_k (k\ge 1)$, $u$ is orthogonal to the constant function $u_0=1$ in the $L^2(\partial \Omega)$ sense.
	If $u|_{\partial\Omega}=c$ is constant, then
	\begin{equation*}
		0=\int_{\partial\Omega} u=cL(\partial\Omega),
	\end{equation*}
	which implies $c=0$. Hence, $u\equiv 0$ in $\Omega$ by the harmonicity of $u$, a contradiction.
\end{proof}

The following lemma allows us to construct an auxiliary function in the next sections.
\begin{lem}[cf.~\cite{PP91}*{Lemma 3} and \cite{LS25}*{Lemma 2.1}]\label{lem-2.3}
	Let $\Omega$ be a connected compact surface with $C^2$ boundary. Then any first Steklov eigenfunction has no critical point in $\Omega$.
\end{lem}

We end this section with a statement on the conformal invariance of solvability.
\begin{lem}\label{lem-conf}
Let $(\Omega, g)$ be a connected compact Riemannian surface with boundary $\partial\Omega$,
For any conformal metric $\tilde{g}=\rho^2 g$,
if $\rho|_{\partial\Omega}$ is constant, then
$(\Omega,g)$ is solvable for \eqref{PP-I} if and only if $(\Omega,\tilde{g})$ is solvable for \eqref{PP-I}.

In other words, if $(\Omega_1,g_1)\sim (\Omega_2,g_1)$,
then the solvability of $\Omega_1$ and $\Omega_2$ for \eqref{PP-I} are the same.
\end{lem}
\begin{proof}
	Assume $\rho|_{\partial\Omega}\equiv \rho_0>0$.
	Since $\nabla^{\tilde{g}} u =\rho^{-2} \nabla^{g} u,\, \nu_{\tilde{g}}=\rho^{-1}\nu_{g}$, we conclude
	\begin{gather*}
		\Delta_{\tilde{g}} u=0\iff \Delta_{g} u=0 \text{ in } \Omega,\\
		\langle \nabla^{\tilde{g}} u, \nu_{\tilde{g}}\rangle_{\tilde{g}}=\rho_{0}^{-1}\langle \nabla^{g} u, \nu_{g}\rangle_{g} \text{ on } \partial \Omega, \quad |\nabla^{\tilde{g}} u|^2_{\tilde{g}}=\rho_{0}^{-2}|\nabla^{g} u|^2_{g} \text{ on } \partial \Omega,
	\end{gather*}
	where we used $\Delta_{\tilde{g}}=\rho^{-2}\Delta_{g}$  for the case of dimension two.
	In fact, we have seen that the Steklov spectrum of $(\Omega,\tilde{g})$ is the Steklov spectrum of $(\Omega,g)$ multiplying $\rho_0^{-1}$ with the same Steklov eigenfunctions.
\end{proof}

\section{The case of simply connected surfaces}\label{Sect:simply}
In this section, we will prove the following theorem.
\begin{thm}\label{thm-sim-conn}
	Let $(\Omega,g_\Omega)$ be a simply connected compact surface with boundary $\partial \Omega$. Then $\Omega$ is solvable for the overdetermined problem \eqref{PP-I} if and only if $\Omega\sim \mathbb{D}$.
\end{thm}

We introduce an auxiliary function $f=\log |\nabla u|$  (cf.~\cites{PP91, LS25}),
where $u$ is any first Steklov eigenfunction of $\Omega$.
Due to Lemma~\ref{lem-2.3}, $f$ is well-defined.
We need to analyze $f$ carefully.

\begin{lem}[cf.~\cite{LS25}*{Lemma 2.2}]\label{lem-1}
	Assume $u$ is a harmonic function with no critical points in $\Omega$.
	Then the Laplacian of $f=\log |\nabla u|$ satisfies
	\begin{equation*}
		\Delta f= K \text{ in } \Omega,
	\end{equation*}
	where $K$ is the Gaussian curvature of $\Omega$.
\end{lem}
\begin{proof}
	This has been proven in \cite{LS25}. Here we give a quick proof by using the moving frame. We will see that the dimension $2$ is crucial.

	We choose a local orthonormal frame $\{e_1, e_2\}$ on $\Omega$, and denote $\nabla u=u_1e_1+u_2e_2$, $\Hess u(e_i,e_j)=u_{ij}$. Without loss of generality, we can assume $(u_{ij})$ is diagonal. Hence, we have $u_{11}+u_{22}=\Delta u=0, u_{12}=u_{21}=0$.
	The Ricci identity gives
	\begin{equation*}
		\sum_{i,j} u_i u_{ijj}=\sum_{i,j} u_i u_{jji}+\sum_{i,j,k} u_i u_k R_{kjij}=u_1^2R_{1212}+u_2^2R_{2121}=|\nabla u|^2K.
	\end{equation*}
	Now a direct computation yields
	\begin{align*}
		\Delta f&=\frac{1}{2}\log|\nabla u|^2=\frac{1}{2}\frac{\Delta |\nabla u|^2}{|\nabla u|^2}-\frac{1}{2}\frac{|\nabla |\nabla u|^2|^2}{|\nabla u|^4}\\
		&=\frac{\sum_{i,j}(u_i u_{ijj}+u_{ij}^2)}{|\nabla u|^2}-\frac{2\sum_{j}(\sum_{i}u_i u_{ij})^2}{|\nabla u|^4}\\
		&=K+\frac{u_{11}^2+u_{22}^2}{|\nabla u|^2}-\frac{2(u_1^2u_{11}^2+u_2^2u_{22}^2)}{|\nabla u|^4}\\
		&=K+\frac{2u_{11}^2}{|\nabla u|^2}-\frac{2(u_1^2+u_2^2)u_{11}^2}{|\nabla u|^4}\\
		&=K.
	\end{align*}
	This completes the proof.
\end{proof}

We don't need \eqref{eq-s3} when applying the above lemma to the first Steklov eigenfunction $u$.
However, this additional condition is necessary in the key proposition below.
\begin{prop}\label{prop-2}
	Let $\Omega$ be a simply connected surface and $u$ be a first Steklov eigenfunction satisfying \eqref{eq-s3}.
	Then the normal derivative of $f=\log |\nabla u|$  on the boundary satisfies
	\begin{equation*}
		\frac{\partial f}{\partial \nu}=\sigma_1-\kappa \text{ on } \partial \Omega,
	\end{equation*}
	where $\kappa$ is the geodesic curvature of $\partial \Omega$.
\end{prop}

\begin{proof}
	Denote $F=\{p\in \partial\Omega \mid \frac{\partial f}{\partial \nu}=\sigma_1-\kappa \text{ at } p\}$. It is sufficient to prove $F=\partial\Omega$.

	Clearly, $F$ is closed while its complement set
	$F^c=\partial\Omega\backslash F$ is open.
	Consider the open set $G=\{p\in \partial\Omega \mid u_y(0,y)\neq 0 \text{ at } p=\gamma(y)\}$. We make the following assertions.

	\textbf{Claim 1}: $G$ is not empty.

	\textbf{Claim 2}: $\overline{G}(=\supp u_y|_{\partial\Omega})\subset F$.

	\textbf{Claim 3}: if $F^c$ is not empty, then $\frac{\partial f}{\partial \nu}=-\kappa$ on $F^c$.

	By Claims 1 and 2, $F$ is not empty.
	If $F^c$ is not empty, then $F^c$ is not closed, otherwise $F^c=\partial\Omega$ by the connectedness of $\partial\Omega$, which contradicts $F\neq \emptyset$. Hence, there is a point $p\in \overline{F^c}\cap F$. By Claim 3 and the continuity, we have
	$\frac{\partial f}{\partial \nu}=-\kappa$ on $F^c$ at $p$ ; on the other hand, $\frac{\partial f}{\partial \nu}=\sigma_1-\kappa$ at $p$ by the definition of $F$. This is impossible.
	Therefore, $F^c$ must be empty, i.e., $F=\partial\Omega$.

	\smallskip
	Now we prove Claims 1--3 to complete the whole proof.
	We must carefully distinguish between taking derivatives on the boundary and taking derivatives within the domain and then restricting them to the boundary. Although all the calculations can be done via the moving frame, here we adopt the local coordinates in Sect. \ref{Section:nep} to see it more clearly.

	Take $\phi=u$ in \eqref{eq2.3}, we have
	\begin{equation*}
		0=\Delta u=u_{xx}+\frac{1}{\lambda^{2}}u_{yy}+\frac{\lambda_x}{\lambda}u_x-\frac{\lambda_y}{\lambda^3}u_y.
	\end{equation*}
	Restricted this to the boundary, noticing $\lambda(0,y)\equiv 1$, we have
	\begin{equation}\label{eq2.5}
		(u_{xx}+u_{yy}+\lambda_x u_x)|_{x=0}=0.
	\end{equation}

	The boundary condition \eqref{eq-s3} is equivalent to
	\begin{equation}\label{eq-s3''}
		u_x^2(0,y)+u_y^2(0,y)=c^2,
	\end{equation}
	where we used \eqref{eq2.2} and $\lambda(0,y)=1$. Differentiating \eqref{eq-s3''} with respect to $y$, we obtain
	\begin{equation}\label{eq2.6}
		(u_x u_{xy}+u_y u_{yy})|_{x=0}=0.
	\end{equation}

	From $\frac{\partial u}{\partial \nu}=\langle \nabla u, -\partial_x\rangle=\sigma_1 u$ on $\partial \Omega$, noticing \eqref{eq2.1}, we derive
	\begin{equation}\label{eq-s2'}
		-u_x(0,y)=\sigma_1 u(0,y)
	\end{equation}
	and then
	\begin{equation}\label{eq2.7}
		u_{xy}(0,y)=-\sigma_1 u_y(0,y).
	\end{equation}

	Similarly,
	\begin{equation}\label{eq2.8}
		\frac{\partial f}{\partial \nu}=-f_x|_{x=0}=-
		\left.\frac{1}{2}\frac{|\nabla u|^2_x}{|\nabla u|^2}\right|_{x=0}.
	\end{equation}
	By using \eqref{eq2.2}, we have
	\begin{equation}\label{eq2.9}
		\frac{1}{2}|\nabla u|^2_x\Big|_{x=0}=(u_x u_{xx}+\frac{1}{\lambda^2}u_y u_{yx}-\frac{\lambda_x}{\lambda^3}u_y^2)\Big|_{x=0}
		=(u_x u_{xx}+u_y u_{yx}-\lambda_x u_y^2)\Big|_{x=0}.
	\end{equation}

	If $G$ is empty, then $u_y(0,y)\equiv 0$, which implies $u|_{\partial\Omega}$ is constant and then contradicts Lemma~\ref{lem-2.2}. Claim 1 is proved.

	Inserting \eqref{eq2.7} into \eqref{eq2.6}, we have
	\begin{equation}\label{eq2.10}
		u_y(0,y)(u_{yy}(0,y)-\sigma_1 u_x(0,y))=0.
	\end{equation}
	When $u_y(0,y)\neq 0$, from \eqref{eq2.10} we have
	\begin{equation}\label{eq2.11}
		u_{yy}(0,y)=\sigma_1 u_x(0,y).
	\end{equation}
	Combining \eqref{eq2.11} with \eqref{eq2.5}, we conclude
	\begin{equation}\label{eq2.12}
		u_{xx}|_{x=0}=(-\sigma_1 u_x-\lambda_x u_x)|_{x=0}.
	\end{equation}
	Hence, \eqref{eq2.9} becomes
	\begin{align*}
		\frac{1}{2}|\nabla u|^2_x\Big|_{x=0}&=
		\big(u_x(-\sigma_1 u_x-\lambda_x u_x)-\sigma_1 u_y u_y-\lambda_x u_y^2\big)\Big|_{x=0}\\
		&=-(\sigma_1+\lambda_x)(u_x^2+u_y^2)\Big|_{x=0}\\
		&=-(\sigma_1+\lambda_x)|\nabla u|^2\Big|_{x=0}
	\end{align*}
	by using \eqref{eq2.7} and \eqref{eq2.12}.
	Inserting this equation into \eqref{eq2.8} and using \eqref{eq-geo-cur}, we obtain
	\begin{equation*}
		\frac{\partial f}{\partial \nu}=\sigma_1-\kappa.
	\end{equation*}
	This shows $G\subset F$ and Claim 2 is obtained by the continuity.

	If $F^c$ is not empty, since $F^c$ is open, we have
	$u_y(0,y)=0$ and $u_{yy}(0,y)=0$ on $F^c$.
	From \eqref{eq2.5} we have $u_{xx}|_{x=0}=-(\lambda_x u_x)|_{x=0}$, then \eqref{eq2.9} becomes
	\begin{equation*}
		\frac{1}{2}|\nabla u|^2_x\Big|_{x=0}=-\lambda_x u_x^2\Big|_{x=0}
		=\kappa|\nabla u|^2.
	\end{equation*}
	Inserting this into \eqref{eq2.8}, we obtain $\frac{\partial f}{\partial \nu}=-\kappa$ on $F^c$. Claim 3 is proved.
\end{proof}

Now we are ready to prove the theorem.
\begin{proof}[Proof of Theorem~\ref{thm-sim-conn}]
	Assume the solvability of \eqref{PP-I},
	by using Lemma~\ref{lem-1}, Proposition~\ref{prop-2}, the Stokes formula and the Gauss--Bonnet formula, we have
	\begin{align*}
		\int_{\Omega}K &= \int_{\Omega}\Delta f= \int_{\partial \Omega}  \frac{\partial f}{\partial \nu}\\
		& =\int_{\partial \Omega}\sigma_1-\int_{\partial \Omega}\kappa\\
		& =\sigma_1 L(\partial \Omega)+\int_{\Omega}K-2\pi,
	\end{align*}
	Hence, $\sigma_1 L(\partial \Omega)=2\pi$.
	The famous Weinstock inequality (cf.~\cites{Wei54,FS11}) states $\sigma_1 L(\partial \Omega)\le 2\pi$ with equality if and only if there is conformal map from $\Omega$ to the unit disk which is an isometry on the boundary.

	Conversely, the following example shows that the disk is solvable.
	Hence, we complete the proof.
\end{proof}

\begin{ex}\label{ex-disk}
	Denote by $D_K(R)$ the geodesic disk of radius $R$  in $\mathbb{M}^2_K$.
	Then $D_K(R)$ is solvable for \eqref{PP-I}.

	Indeed, since $\sigma_1(\mathbb{D})=1$ and with multiplicity $2$, and the first eigenspace is spanned by $\{u_1^{(1)}=r\cos \theta, u_1^{(2)}=r\sin \theta\}$ (under the polar coordinates). One can easily verify that $|\nabla u_1^{(i)}|=1$ in the whole $\mathbb{D}$ including the boundary.
	Based on the fact $D_K(R)\sim \mathbb{D}$, we finally obtain the conclusion by using Lemma~\ref{lem-conf}.

	In fact, the solvability of $D_K(R)$ can also be directly verified via the explicit expression of the first Steklov eigenspace of $D_K(R)$ (cf.~\cites{BS14, Xio22}).
\end{ex}

By considering the equality case of the Weinstock inequality for domains in space forms,
we immediately obtain the following
\begin{cor}\label{cor-sim-conn}
	A simply connected compact domain $\Omega\subset\mathbb{M}_K^2$ is solvable for the overdetermined problem $\eqref{PP-I}$ if and only if $\Omega$ is a geodesic disk.
\end{cor}
Clearly, Corollary~\ref{cor-sim-conn} recovers Theorem~\ref{thm:PP1} of Payne and Philippin when $K=0$.
It is new when $K=1$ or $-1$.

\section{The case of multiply connected surfaces}\label{Sect:multiply}
In this section, we are concerned with multiply connected surfaces.
We first obtain the characterization of Riemannian surfaces,
and then apply the result to the domains in space forms.
\subsection{Riemannian surfaces}
We will prove the following theorem.
\begin{thm}\label{thm-multi-conn}
	Let $(\Omega,g_\Omega)$ be a multiply connected  compact surface.
	Then $\Omega$ is solvable for the overdetermined problem \eqref{PP-I} if and only if $\Omega\sim \mathbb{A}_T$ for some $T\ge T_1$, where $\mathbb{A}_T$ and $T_1$ are defined in Sect.~\ref{sect-intro}.
\end{thm}

First, let us check the solvability of the cylinder $\mathbb{A}_T$.
It is a special case of \cite{GP17}*{Example 1.3.3} (originally from \cite{CESG11}*{Lemma 6.1}) and shows that both Proposition~\ref{prop-2} and Theorem~\ref{thm-sim-conn} no longer hold if the simply connected condition is removed.
\begin{ex}\label{ex-annulus}
	The boundary $\partial\mathbb{A}_T$ has two connected components $\Gamma_{\pm}\coloneq\{\pm T\}\times S^1$.

	By the method of separation of variables, the Steklov spectrum of $\mathbb{A}_T$ is given by
	\begin{equation*}
		0, \quad 1/T, \quad k\tanh(kT), \quad k\coth(kT)\quad  (k\ge 1),
	\end{equation*}
	and the corresponding eigenfunctions are
	\begin{equation*}
		1, \quad t, \quad \cosh(kt)v_k(\theta), \quad \sinh(kt)v_k(\theta) \quad  (k\ge 1),
	\end{equation*}
	where $v_k(\theta)=\cos(k\theta)$ or $\sin(k\theta)$ is the eigenfunction of the Laplacian operator on $S^1$ corresponding to the eigenvalue $\lambda_k(S^1)=k^2$.

	Since $\tanh(kT)\le \coth(kT)$ for any $k\ge 1$ and $\tanh(\cdot)$ is strictly increasing, we have $\sigma_1=1/T$ or $\sigma_1=\tanh T$.
	By the monotonicity of $1/T$ and $\tanh T$ with respect to $T$ and the asymptotic behaviors as $T$ tends to $0$ and $\infty$, there is a unique constant $T_1>0$ such that $1/T_1=\tanh T_1$.

	(1) When $T\ge T_1$, we have $1/T\le \tanh T$. Hence, $\sigma_1=1/T$.
	Moreover, $\sigma_1$ is simple when $T> T_1$ while $\sigma_1$ has multiplicity $3$ when $T=T_1$.
	In this case, $u_1=t$ solves \eqref{PP-I}, since $\nabla u_1=\partial_t$ and then $|\nabla u_1|\equiv 1$ on $\mathbb{A}_T$.

	Moreover, $u_1\equiv T$ on $\Gamma_{+}$ while $u_1\equiv -T$ on $\Gamma_{-}$, i.e., $u_1$ is constant on each connected component of the boundary respectively but is not constant on the whole boundary. Meanwhile, $\frac{\partial}{\partial \nu}\log |\nabla u_1|=0=-\kappa$, which coincides with Claim 3 in the proof of Proposition~\ref{prop-2}, but Proposition~\ref{prop-2} itself fails.

	(2) When $T<T_1$, we have $\tanh T<1/T$. Hence, $\sigma_1=\tanh T$, which has multiplicity $2$. In this case, any first Steklov eigenfunction can be written as
	\begin{equation*}
		u_1=(A\cos\theta+B\sin\theta)\cosh t.
	\end{equation*}
	Since
	\begin{equation*}
		|\nabla u_1|^2= (A\cos\theta+B\sin\theta)^2\sinh^2 T+ (B\cos\theta-A\sin\theta)^2\cosh^2 T\eqcolon Q(\theta)
	\end{equation*}
	on the boundary,
	if $Q(\theta)$ is constant, then
	\begin{equation*}
		A^2\sinh^2 T +B^2\cosh^2 T =Q(0)=Q(\pi/2)= B^2 \sinh^2 T+ A^2\cosh^2 T
	\end{equation*}
	implies $A^2=B^2$.
	Without loss of generality, we take $A=B=1$.
	However,
	\begin{equation*}
		Q(\pi/4)=2\sinh^2 T< \cosh^2 T =Q(-\pi/4),
	\end{equation*}
	which means that $|\nabla u_1|$ is not constant on the boundary.
\end{ex}

In summary, we have proved the following result.
\begin{prop}\label{prop-cylinder}
	The cylinder $\mathbb{A}_T$ is solvable for the overdetermined problem \eqref{PP-I} if and only if $T\ge T_1$, where $T_1$ is the unique positive solution of the equation $\tanh x=1/x$.
\end{prop}

Example~\ref{ex-annulus} motivates us to give a modified version of Proposition~\ref{prop-2}.
\begin{prop}\label{prop-3}
	Let $(\Omega, g_{\Omega})$ be a connected orientable compact surface.
	Assume $\partial \Omega$ has $b$ connected components, denoted by $\Gamma_1,\cdots, \Gamma_b (b\ge 1)$.
	Let $u$ be a nontrivial solution of \eqref{PP-I} and $f=\log |\nabla u|$.
	For each $i$, we have the following assertions.

	(1) If $u$ is constant on $\Gamma_i$, then
	\begin{equation*}
		\frac{\partial f}{\partial \nu}=-\kappa \text{ on } \Gamma_i.
	\end{equation*}

	(2) If $u$ is not constant on $\Gamma_i$, then
	\begin{equation*}
		\frac{\partial f}{\partial \nu}=\sigma_1-\kappa \text{ on } \Gamma_i.
	\end{equation*}

	Moreover, Assertion (1) does not occur when $b=1$.
\end{prop}
\begin{proof}
	It suffices to replace $\partial \Omega$ with $\Gamma_i$ in the proof of Proposition~\ref{prop-2} and repeat the argument.
	If $u$ is constant on $\Gamma_i$, then $u_y(0,y)\equiv 0$ on $\Gamma_i$, and Assertion (1) follows from Claim 3.

	If $u$ is not constant on $\Gamma_i$, then $G$ is not empty, and Assertion (2) follows from the same reasoning as in the proof of Proposition~\ref{prop-2}.
\end{proof}

\begin{proof}[Proof of Theorem~\ref{thm-multi-conn}]
	The sufficiency follows directly from Proposition~\ref{prop-cylinder} and Lemma~\ref{lem-conf}.
	Next we focus on establishing the necessity.

	\textbf{(I) Orientable Case.}
	Assume the solvability of \eqref{PP-I} for an orientable surface $\Omega$.
	By Lemma~\ref{lem-1}, Proposition~\ref{prop-3} and Gauss--Bonnet formula, we have
	\begin{align*}
		\int_{\Omega}K &= \int_{\Omega}\Delta f= \int_{\partial \Omega} \frac{\partial f}{\partial \nu}
		=\sum_{i=1}^b\int_{\Gamma_i}\delta_i\sigma_1-\int_{\partial \Omega}\kappa\\
		& =\sum_{i=1}^b\sigma_1 \delta_i L(\Gamma_i)+\int_{\Omega}K-2\pi\chi(\Omega),
	\end{align*}
	where $\delta_i=0$ or $1$, depending on whether $u|_{\Gamma_i}$ is constant or not.
	This leads to
	\begin{equation}\label{eq-3.12}
		\sigma_1 \sum_{i=1}^b \delta_i L(\Gamma_i)=2\pi\chi(\Omega)=2\pi (2-2\gamma-b),
	\end{equation}
	where $\gamma$ denotes the genus of $\Omega$.

	Since the left-hand side of \eqref{eq-3.12} is nonnegative, \eqref{eq-3.12} cannot hold if  $\gamma\ge 1$ or $b\ge 3$.
	Thus, we must have $\gamma=0$ and $b\le 2$.
	We rule out the case $\gamma=0, b=1$ since  $\Omega$ would  be simply connected.

	When $\gamma=0$ and $b=2$, $\Omega$ is topologically an annulus.
	The right-hand side of \eqref{eq-3.12} is equal to $0$, which forces $\delta_1=\delta_2=0$.
	By Proposition~\ref{prop-3}, $u$ must be constant on each $\Gamma_i$, denoted by $c_i$.
	Then we deduce $c_i^2=c^2/\sigma_1^2$ on $\Gamma_i$ from \eqref{eq-s3''} and \eqref{eq-s2'} with $u_y(0,y)$.
	On the other hand, $\int_{\partial\Omega}u=0$ implies $
		c_1L(\Gamma_1)+c_2L(\Gamma_2)=0$,
	hence, $c_1=-c_2$ and $L(\Gamma_1)=L(\Gamma_2)$.

	Now by Proposition~\ref{prop-annulus-homothetic} (see below) and Example~\ref{ex-annulus},
	we obtain that $\Omega$ is $\sigma$-homothetic to $\mathbb{A}_T$ for some $T\ge T_1$.

	\textbf{(II) Non-orientable case.}
	Assume that $(\Omega, g)$ is non-orientable.
	We consider its orientable double cover $\tilde{\Omega}$ with the $2:1$ projection $\pi \colon \tilde{\Omega}\to \Omega$ and the metric $\tilde{g}=\pi^{\ast}g$.
	Let $\tilde{u}=\pi^{\ast}u=u\circ \pi$.
	Clearly, if $u$ has no critical points in $\Omega$,
	then $\tilde{u}$ has no critical points in $\tilde{\Omega}$.
	If $u$ solves \eqref{PP-I} for $\Omega$,
	then $\tilde{u}$ solves \eqref{PP-I} for $\tilde{\Omega}$.
	Although $\sigma_1=\sigma_1(\Omega)$ may not be the first nonzero Steklov eigenvalue of $\tilde{\Omega}$,
	Lemma~\ref{lem-1} and Proposition~\ref{prop-3} are still valid.
	Similar to the proof of Theorem~\ref{thm-multi-conn},
	we obtain $\tilde{\Omega}$ is topologically an annulus, and $\tilde{u}$ takes opposite constant values on the two boundary components. This is a contradiction.

	This completes the whole proof.
\end{proof}

We now prove the following proposition used in the proof of Theorem~\ref{thm-multi-conn}.
\begin{prop}\label{prop-annulus-homothetic}
	Let $\Omega$ be topologically an annulus equipped with a smooth metric $g_\Omega$.
	If $\Omega$ has a Steklov eigenfunction $u$ such that the restriction of $u$ to the two boundary components are constants of opposite sign,
	then there exists a $\sigma$-homothety $\Phi \colon \Omega\to \mathbb{A}_T$ for some $T>0$, and $u$ is $\Phi^{\ast}t$ times a nonzero constant.
\end{prop}

\begin{proof}
	By complex analysis, there exists a conformal diffeomorphism
	\begin{equation*}
		\Phi \colon \Omega \to \mathbb{A}_T
	\end{equation*}
	for some $T>0$.
	Without loss of generality, we assume $\Phi(\Gamma_1)=\Gamma_{-}$ and $\Phi(\Gamma_2)=\Gamma_{+}$.

	By the conformal map $\Phi$, we can parametrize $\Omega$ by $(t,\theta)\in [-T,T]\times S^1$ and write the metric as $g_\Omega=\rho^2(t,\theta)(\mathrm{d}t^2+\mathrm{d}\theta^2)$.

	Suppose $u$ is a Steklov eigenfunction such that $-u|_{\Gamma_1}=u|_{\Gamma_2}=c$, where $c$ is a nonzero constant.
	Then $u$ is the unique solution of the following
	boundary value problem
	\begin{equation*}
	\begin{cases}
		u_{tt}+u_{\theta\theta}=0,\\
		u(-T,\theta)=-c, u(T,\theta)=c.
	\end{cases}
\end{equation*}
	We can directly verify that the linear function $u=ct/T$ is the unique solution.

	Since $u$ is a nonconstant Steklov eigenfunction, we have
	$\frac{\partial u}{\partial\nu}=\sigma u$ on $\partial \Omega$, where $\sigma>0$, $\nu=-\rho^{-1} \partial_t$ on $\Gamma_{1}$, and $\nu=\rho^{-1} \partial_t$ on $\Gamma_{2}$.
	This implies
	\begin{equation*}
		-\rho^{-1}(-T,\theta)\frac{c}{T}=-c\sigma,\quad \rho^{-1}(T,\theta)\frac{c}{T}=c\sigma.
	\end{equation*}
	Hence, $\rho=1/(\sigma T)$ is constant on $\partial\Omega$ since $c\neq 0$,
	which implies $\Phi$ is a $\sigma$-homothety.
\end{proof}

\subsection{Domains in space forms}
In the rest of this section, we investigate the solvability of domains in space forms.
By applying Theorem~\ref{thm-multi-conn}, we prove the following results.

\begin{cor}\label{cor-multi-conn-euc-hyp}
	Any multiply connected compact domain $\Omega\subset \mathbb{R}^2$ or $\Omega\subset\mathbb{H}^2$ with boundary $\partial \Omega$ is unsolvable for $\eqref{PP-I}$.
\end{cor}
\begin{proof}
	Suppose $\Omega\subset \mathbb{M}^2_K$ is solvable for \eqref{PP-I},
	then $\Omega \sim \mathbb{A}_T$ for some $T$ by Theorem~\ref{thm-multi-conn}.
	We write the metric $g$ of $\Omega$ as $g=\rho^2 g_0$ with $g_0=\mathrm{d}t^2+\mathrm{d}\theta^2$,
	then
	\begin{equation}\label{eq-Gauss-PDE}
		e^{-2w}\Delta_0 w= -K,
	\end{equation}
	where $w=\ln \rho$ and $\Delta_0=\frac{\partial^2}{\partial t^2} + \frac{\partial^2}{\partial \theta^2}$ is the Laplacian with respect to $g_0$.

	(1) $\Omega\subset \mathbb{R}^2~(K=0)$:
	The equation \eqref{eq-Gauss-PDE} becomes $\Delta_0 w=0$, i.e., $w$ is harmonic.
	Since $w|_{\partial\Omega}$ is constant, we deduce that both $w$ and $\rho=e^{w}$ are  constant.
	Thus, $\Omega$ is isometric to $\mathbb{A}_T$ after rescaling.
	It is a contradiction, since each connected component of $\partial\mathbb{A}_T$ is a geodesic, but each complete geodesic in $\mathbb{R}^2$ is a straight line and it is non-compact.

	(2) $\Omega\subset \mathbb{H}^2~(K=-1)$:
	The equation \eqref{eq-Gauss-PDE} becomes $\Delta_0 w=e^{2w}$.
	Since $w|_{\partial \Omega}$ is constant,
	this semi-linear elliptic equation has at most one solution by the maximum principle.
	If such a solution exists, it is independent of $\theta$ and is even in $t$,
	i.e., $w=w(t) $ and $w(t)=w(-t)$; so is $\rho=e^w$.
	Consequently, the geodesic curvature of each boundary component is constant.
	Since the boundary components have constant geodesic curvature and equal length,
	both $\Gamma_1$ and $\Gamma_2$ are closed geodesic circles of the same radius.
	However, such circles cannot be nested and do not bound a compact domain.
	Hence, it is impossible that $\Omega\sim \mathbb{A}_T$.

	This completes the proof.
\end{proof}

\begin{cor}\label{cor-multi-conn-sph}
	A multiply connected domain $\Omega\subset\mathbb{S}^2$ with boundary $\partial \Omega$ is solvable for $\eqref{PP-I}$ if and only if $\Omega\sim \sband(R)$ for some $R\in [R_1,\pi/2)$.
	Moreover, $\Omega$ is symmetric with respect to a certain great circle of $\mathbb{S}^2$.
\end{cor}

\begin{proof}
	The metric of $\sband$ is
	\begin{equation*}
	g_1=\mathrm{d}r^2+\cos^2(r)\mathrm{d}\theta^2.
	\end{equation*}
Under the change of variables
\begin{equation*}
	t=\ln \frac{1+\tan(r/2)}{1-\tan(r/2)} \iff r=2\arctan \tanh(t/2),
\end{equation*}
we obtain
\begin{equation*}
	g_1=\frac{1}{\cosh^2(t)}(\mathrm{d}t^2+\mathrm{d}\theta^2).
\end{equation*}
This implies $\sband(R) \sim \mathbb{A}_T$ with $R=2\arctan \tanh(T/2)$.
Moreover, $t=0 \iff r=0$ while $T_1\le T< +\infty \iff R_1\le R<\pi/2$.

Hence, $\Omega$ is solvable for \eqref{PP-I} if and only if $\Omega \sim \mathbb{A}_T\sim \sband(R)$ for some $R\in [R_1,\pi/2)$ by Theorem~\ref{thm-multi-conn}.

When $K=1$, the equation \eqref{eq-Gauss-PDE} becomes $\Delta_0 w=-e^{2w}$.
The boundary condition is $w(T,\cdot)=w(-T,\cdot)=c$.
By using the moving plane method (cf.~the proof of Theorem 2.1 in \cite{GNN79}), we can prove that the solution $w$ of this boundary value problem is symmetric with respect to $t$, and $\frac{\partial w}{\partial t}>0$ for $t\in (-T,0)$ while $\frac{\partial w}{\partial t}<0$ for $t\in (0,T)$.
Therefore, $\rho=e^w$ is also symmetric with respect to $t$ and $\rho_t(0,\cdot)=0$.
It follows that the image of $\{0\}\times S^1$ under a $\sigma$-homothety $\Phi\colon \mathbb{A}_T\to \Omega$ is a closed geodesic of $\mathbb{S}^2$.
Hence, $\Omega$ is symmetric with respect to a certain great circle of $\mathbb{S}^2$.
\end{proof}

\begin{rem}
	When $K=1$, the solution to $\Delta_0 w=-e^{2w}$ with $w(T,\cdot)=w(-T,\cdot)=c$ is not necessarily unique, as a bifurcation phenomenon may occur.
	However, if the solution is independent of $\theta$, then the corresponding $\Omega$ is isometric to $\sband(R)$ for some $R$.
\end{rem}

\section{Further discussion}\label{Sect:general-case}
By examining the arguments in Sects.~\ref{Sect:simply} and \ref{Sect:multiply},
one can see that both Lemma~\ref{lem-1} and Proposition~\ref{prop-3} remain valid when replacing $\sigma_1$ with any nonzero Steklov eigenvalue $\sigma$,
provided that $f=\log|\nabla u|$ is well-defined.
Moreover, the proof of Proposition~\ref{prop-3} requires only that $|\nabla u|$ is constant on each boundary component (possibly taking different values on different components); this is a more natural assumption for surfaces with multiple boundary components.

\subsection{A weaker overdetermined condition}
We replace \eqref{eq-s3} with a weaker condition
\begin{equation}\label{eq-s3'}
	|\nabla u| \text{ is locally constant on } \partial\Omega, \tag{1.3'}
\end{equation}
and consider the following overdetermined problem
\begin{equation}\label{PP-I'}
	\left\{
	\begin{aligned}
		\Delta u&=0 && \text{ in } \Omega,\\
		\partial_{\nu}u&=\sigma_1 u && \text{ on } \partial\Omega,\\
		|\nabla u|&=c \text{ is locally constant} && \text{ on } \partial\Omega.
	\end{aligned}\right.\tag{PP-I'}
\end{equation}

The weaker condition yields a broader class of solvable surfaces than the standard cylinder $\mathbb{A}_T$.
Indeed, for any fixed $\alpha\ge 1$, let $\mathbb{A}_{T,\alpha}$ be the cylinder $[-T,T]\times S^1$ equipped with a rotationally symmetric and conformally flat metric $g_{\alpha}=\rho^2_{\alpha}(t)(\mathrm{d}t^2+\mathrm{d}\theta^2)$, where $\rho_{\alpha}(t)=\frac{\alpha-1}{2T}t+\frac{\alpha+1}{2}$ is linear in $t$ and satisfies $\rho(-T)=1$ and $\rho(T)=\alpha$.
Geometrically, this means $\mathbb{A}_{T,\alpha}$ is a lateral surface of a conical frustum (including boundary) with a boundary length ratio $L(\Gamma_{+})/L(\Gamma_{-})=\alpha$.
It is known that
$\mathbb{A}_{T,\alpha}$ admits the linear function $\ell(t)=\frac{1}{T}t+\frac{1-\alpha}{1+\alpha}$ as a Steklov eigenfunction, with the corresponding eigenvalue $\sigma_{\ell} \coloneqq \frac{1+\alpha}{2\alpha T}$.
Furthermore, there exists a critical value $T_{\alpha}$ such that $\sigma_{\ell}=\sigma_1$ if and only if $T\ge T_{\alpha}$.
Specifically, $T_{\alpha}$ is the unique positive solution of
\begin{equation}\label{eq-critical-eq}
	\frac{4\alpha x}{(\alpha+1)^2}=\coth(2x)+\left(\coth^2(2x)-\frac{4\alpha}{(\alpha+1)^2}\right)^{1/2}.
\end{equation}
In particular, $\mathbb{A}_{T,\alpha}$ reduce to $\mathbb{A}_{T}$ and \eqref{eq-critical-eq} becomes $x=\coth x$ when $\alpha=1$.
Consequently, $T_1$ coincides with the previously established definition.
We refer the reader to \cite{FS11}*{Section3} for more details.

Based on the analysis at the beginning of this section, we establish the following theorem, omitting the proof details for brevity.
\begin{thm}
	Let $(\Omega,g_\Omega)$ be a multiply connected  compact surface.
	Then $\Omega$ is solvable for the overdetermined problem \eqref{PP-I'} if and only if $\Omega\sim \mathbb{A}_{T,\alpha}$ for some $T\ge T_{\alpha}$.
\end{thm}
\begin{rem}
	If $|\nabla u|\equiv c_i$ on the boundary component $\Gamma_i~(i=1,2)$,
	then the boundary length ratio $\alpha=L(\Gamma_{2})/L(\Gamma_{1})$ is given by $c_{1}/c_{2}$,
	which follows from $|\nabla u|=\sigma |u|$ on $\partial \Omega=\Gamma_1\cup \Gamma_2$ and $\int_{\partial \Omega} u=0$.
	To ensure $\alpha\ge 1$, we assume $|c_{1}|\ge |c_{2}|$ without loss of generality.
\end{rem}

By standard conformal maps,
a multiply connected compact domain $\Omega\subset \mathbb{R}^2$ or $\Omega\subset\mathbb{H}^2$ with boundary $\partial \Omega$ is solvable for $\eqref{PP-I'}$ if and only if $\Omega$ is an annulus bounded by two concentric geodesic circles with appropriate inner and outer radii.
This stands in contrast to Corollary~\ref{cor-multi-conn-euc-hyp}.

\subsection{Solvability for higher eigenvalues}
Following the same argument as in the proof of Theorem~\ref{thm-multi-conn},
we immediately obtain the following result.
\begin{prop}
	Let $(\Omega,g_\Omega)$ be a multiply connected  compact surface.
	If the overdetermined problem \eqref{eq-s1}, \eqref{eq-s2} and \eqref{eq-s3'}
	admits a solution $u$ with no critical points in $\Omega$,
	then $\Omega$ is $\sigma$-homothetic to $\mathbb{A}_{T,\alpha}$ for some $T>0$ and $\alpha\ge 1$.
	Moreover, any solution of this overdetermined problem is $\Phi^{\ast}\ell$ times a constant, where $\Phi \colon \Omega\to \mathbb{A}_{T,\alpha}$ is a $\sigma$-homothety.

	Furthermore, if \eqref{eq-s3} replaces \eqref{eq-s3'}, then $\alpha$ must be $1$.
\end{prop}

Suppose $u$ has critical points $p_1, \dots, p_k$ in $\Omega$ with multiplicities $m_1, \dots, m_k$, respectively.
Then $\log|\nabla u|$ is well-defined on $\Omega_{\epsilon}\coloneqq\Omega\setminus \bigcup_{i=1}^{k} B(p_i,\epsilon)$,
where $B(p_i,\epsilon)$ denotes the geodesic ball of radius $\epsilon$ centered at $p_i$.
Here, $\epsilon>0$ is chosen sufficiently small such that the balls are pairwise disjoint and compactly contained in $\Omega$.
Let $\mathbf{n}$ denote the outward unit normal to $\partial B(p_i,\epsilon)$,
then the total flux of the gradient of $\log |\nabla u|$ across the boundary satisfies
\begin{equation}\label{eq-dirac}
	\int_{\partial B(p_i,\epsilon)}\frac{\partial}{\partial \mathbf{n}}\log |\nabla u|=2\pi m_i+\int_{B(p_i,\epsilon)}K.
\end{equation}

\begin{rem}
	The identity \eqref{eq-dirac} can be formally understood via the distributional Laplacian of the log-gradient.
	In the sense of distributions,
	$\Delta \log |\nabla u|=2\pi\sum_{i=1}^k m_i\delta_{p_i}+K$ (cf.~Lemma~\ref{lem-1} for the case without critical points),
	where each critical point $p_i$ contributes a Dirac mass of strength $2\pi m_i$.
	For the sake of rigor, we shall provide a direct computational verification at the end of this section.
\end{rem}

Now by using Lemma~\ref{lem-1}, Proposition~\ref{prop-3} and Eq.~\eqref{eq-dirac}, we have
\begin{align*}
	\int_{\Omega_{\epsilon}} K&=
	\int_{\Omega_{\epsilon}}\Delta \log |\nabla u|=\int_{\partial \Omega_{\epsilon}}  \frac{\partial \log |\nabla u|}{\partial \nu}\\
	&=\int_{\partial \Omega}  \frac{\partial \log |\nabla u|}{\partial \nu}-\sum_{i=1}^k
		\int_{\partial B(p_i,\epsilon)} \frac{\partial \log |\nabla u|}{\partial \mathbf{n}}\\
	&=\sigma\sum_{i=1}^{b}\delta_i L(\Gamma_i)-\int_{\partial \Omega} \kappa-\sum_{i=1}^k\int_{B(p_i,\epsilon)}K-2\pi\sum_{i=1}^k  m_i.
\end{align*}
By applying the Gauss--Bonnet formula again, we derive the following necessary condition for the solvability of the overdetermined problem.
\begin{thm}
	Let $(\Omega,g_\Omega)$ be a connected, orientable, compact surface with $b$ boundary components  $\Gamma_1,\dots, \Gamma_b$.
	Set $\delta_i=0$ (resp. $1$) when $u$ is constant (resp. non-constant) on $\Gamma_i$.
	If the overdetermined problem \eqref{eq-s1}, \eqref{eq-s2} and \eqref{eq-s3} (or \eqref{eq-s3'}) admits a nontrivial solution $u$, then
\begin{equation}\label{eq-general-necessary}
	\sigma\sum_{i=1}^{b}\delta_i L(\Gamma_i)=2\pi\big(\chi(\Omega)+m\big),
\end{equation}
where $m$ is the number of the critical points of $u$ in $\Omega$, counted with multiplicities.
\end{thm}

\begin{rem}
	If $\Omega$ is simply connected, then \eqref{eq-general-necessary} reduces to
	\begin{equation*}
		\sigma L(\Omega)=2\pi (m+1),
	\end{equation*}
	which coincides with the result for planar simply connected domains obtained by Alessandrini and Magnanini \cite{AM94}*{Eq.~(2.1)}.
\end{rem}

We conclude this section with the proof of \eqref{eq-dirac}.
Let $(U, \phi)$ be an isothermal coordinate  chart around $p_i$ such that $B(p_i,\epsilon)\subset U$.
The metric on $U$ is expressed as $g=\phi^{\ast}g_0=\rho^2g_0$, where $g_0=\mathrm{d}x^2+\mathrm{d}y^2$.
Then $\tilde{u}=u\circ\phi^{-1}$ is harmonic with respect to $g_0$.
By introducing the complex parameter $z=x+\mathrm{i}y$,
there exists a holomorphic function $H(z)$ such that $\tilde{u}=\Real H(z)$.

Since $z=0$ is the unique critical point of $\tilde{u}$ in the domain $D_{\epsilon}\coloneqq\phi(B(p_i,\epsilon))$ with multiplicity $m_i$,
it follows that $H'(z)=z^{m_i}h(z)$, where $h(z)$ is a holomorphic function with no zeros in $D_{\epsilon}$.
Noting that $|\nabla^0 \tilde{u}|=|H'(z)|$, we have
\begin{equation}\label{eq-5.3}
	\int_{\gamma_{\epsilon}} \frac{\partial \log |\nabla^0 \tilde{u}|}{\partial \mathbf{n}_0}\,|\mathrm{d}z|=\int_{\gamma_{\epsilon}}\frac{\partial \log |z^{m_i}|}{\partial \mathbf{n}_0}\,|\mathrm{d}z|+\int_{\gamma_{\epsilon}}\frac{\partial \log |h(z)|}{\partial \mathbf{n}_0}\,|\mathrm{d}z|,
\end{equation}
where $\gamma_{\epsilon}\coloneqq\phi(\partial B(p_i,\epsilon))$ is a simple closed curve enclosing the origin,
and $\mathbf{n}_0$ is the outward unit normal to $\gamma_{\epsilon}$ in the plane.
The second integral on the right-hand side of \eqref{eq-5.3} vanishes
due to the harmonicity of $\log |h(z)|$ in $D_{\epsilon}$.

For the first integral on the right-hand side of \eqref{eq-5.3}, by the deformation principle for harmonic functions, we can replace $\gamma_{\epsilon}$ with a circle $\{|z|=r\}$ of sufficiently small radius $r$:
\begin{equation*}
	\int_{\gamma_{\epsilon}}\frac{\partial \log |z^{m_i}|}{\partial \mathbf{n}_0}\,|\mathrm{d}z|
	=m_i\int_{|z|=r}\frac{\partial \log |z|}{\partial r}\,|\mathrm{d}z|=m_i \int_{0}^{2\pi}\Big(\frac{1}{r}\Big)\, r\,\mathrm{d}\theta=2\pi m_i.
\end{equation*}

On the other hand, recalling $\Delta_0 \log \rho=-K \rho^2$ and $\mathrm{d}v_g=\rho^2\, \mathrm{d}v_{g_0}$,
we have
\begin{equation*}
	\int_{\gamma_{\epsilon}}\frac{\partial \log \rho}{\partial \mathbf{n}_0}\,|\mathrm{d}z|=\int_{D_{\epsilon}}\Delta_0(\log \rho)\,\mathrm{d}v_{g_0}=-\int_{B(p_i,\epsilon)}K\,\mathrm{d}v_{g}.
\end{equation*}
Consequently, by the conformal invariance of the two-dimensional flux integral, and we derive
\begin{align*}
	\int_{\partial B(p_i,\epsilon)}\frac{\partial \log |\nabla u|}{\partial \mathbf{n}}\,\mathrm{d}s&=\int_{\gamma_{\epsilon}}\frac{1}{\rho}\frac{\partial \log (\rho^{-1}|\nabla^0 \tilde{u}|)}{\partial \mathbf{n}_0}\,\rho|\mathrm{d}z|\\
	&=\int_{\gamma_{\epsilon}}\frac{\partial \log |\nabla^0 \tilde{u}|}{\partial \mathbf{n}_0}\,|\mathrm{d}z|-\int_{\gamma_{\epsilon}}\frac{\partial \log \rho}{\partial \mathbf{n}_0}\,|\mathrm{d}z|\\
	&=2\pi m_i+\int_{B(p_i,\epsilon)}K\,\mathrm{d}v_{g},
\end{align*}
which is the desired result.

\section{Another overdetermined condition}\label{Sect:proof-thmC}
Payne and Philippin \cite{PP91} also considered another additional constraint for the Steklov problem, that is, every nonconstant Steklov eigenfunction has zero mean value:
\begin{equation}\label{eq-s4}
	\int_{\Omega}u_k=0 \text{ for all $k\ge 1$}.
\end{equation}
For simplicity, we denote the overdetermined problem \eqref{eq-s1}, \eqref{eq-s2} and \eqref{eq-s4} by (PP-II).
Payne and Philippin proved the following rigidity result for  plane domains.
\begin{thm}[\cite{PP91}*{Theorem 2}]\label{thm:PP2}
	Let $\Omega\subset \mathbb{R}^2$ be a connected domain with smooth boundary $\partial \Omega$. Then the overdetermined problem (PP-II) is solvable if and only if $\Omega$ is a round disk.
\end{thm}

In fact, from the proof one can see that Theorem~\ref{thm:PP2} holds for any dimension $n\ge 2$. We extend this theorem to the hemisphere and the hyperbolic space.
\begin{thm}\label{thm-C}
	Let $\Omega\subset\mathbb{S}^n_{+}$ or $\Omega\subset\mathbb{H}^n$ be a connected domain of  with smooth boundary $\partial \Omega$.
	Then the overdetermined problem (PP-II) is solvable if and only if $\Omega$ is a geodesic disk.
\end{thm}

\begin{rem}
	Because of the example of $\sband(R)$, we cannot replace $\Omega\subset\mathbb{S}^n_{+}$ with $\Omega\subset\mathbb{S}^n$ while keeping other conditions unchanged.
\end{rem}

The proof of Theorem~\ref{thm-C} relies on the following lemma, which extends \cite{PP91}*{Lemma 4} (see also \cites{Ben86, PS89}) from $\mathbb{R}^n$ to $\mathbb{S}^n_{+}$ and $\mathbb{H}^n$, and admits an analogous proof.
\begin{lem}\label{lem-4.1}
	Let $\Omega\subset\mathbb{S}^n_{+}$ or $\Omega\subset\mathbb{H}^n$ be a connected domain with Lipschitz boundary $\partial \Omega$.
	If
	\begin{equation}\label{eq-4.1}
		\frac{1}{|\Omega|}\int_{\Omega}h=\frac{1}{|\partial\Omega|}\int_{\partial\Omega} h
	\end{equation}
	for any harmonic function $h\in C^2(\Omega)$,
	then $\Omega$ is a geodesic disk.
\end{lem}
\begin{proof}
	Let $u$ be the solution of the Saint-Venant problem in $\Omega$:
	\begin{empheq}[left=\empheqlbrace]{alignat=3}
		\Delta u&=-1 \quad & & \text{ in } \Omega,\label{eq-S1}\\
		u &=0   & & \text{ on } \partial \Omega.\label{eq-S2}
	\end{empheq}
	Then for any harmonic function $h$, we have
	\begin{equation}\label{eq-4.2}
		-\int_{\Omega}h=\int_{\Omega}h\Delta u=\int_{\Omega}u\Delta h+\int_{\partial\Omega}\Big(h\frac{\partial u}{\partial \nu}-u\frac{\partial h}{\partial \nu}\Big)=\int_{\partial\Omega}h\frac{\partial u}{\partial \nu}.
	\end{equation}
	Inserting \eqref{eq-4.2} into \eqref{eq-4.1} yields
	\begin{equation*}
		-\frac{1}{|\Omega|}\int_{\partial\Omega}h\frac{\partial u}{\partial \nu}=\frac{1}{|\partial\Omega|}\int_{\partial\Omega} h,
	\end{equation*}
	equivalently,
	\begin{equation*}
		\int_{\partial\Omega}h\Big(\frac{\partial u}{\partial \nu}+\frac{|\Omega|}{|\partial\Omega|}\Big)=0.
	\end{equation*}
	The arbitrariness of $h$ leads to
	\begin{equation}\label{eq-S3}
		\frac{\partial u}{\partial \nu}=\frac{|\Omega|}{|\partial\Omega|}=\text{constant on $\partial\Omega$},
	\end{equation}
	since we can take a harmonic function $h$ such that $h=\frac{\partial u}{\partial \nu}+\frac{|\Omega|}{|\partial\Omega|}$ on $\partial\Omega$.
	Hence, $u$ solves the overdetermined problem \eqref{eq-S1}, \eqref{eq-S2} and \eqref{eq-S3}.
	By the Serrin-type theorem \cite{KP98} for $\mathbb{S}^2$ and $\mathbb{H}^2$, $\Omega$ must be a geodesic disk.
\end{proof}

\begin{proof}[Proof of Theorem~\ref{thm-C}]
	Let $\{u_i\}_{i\ge 0}$ be a complete orthonormal system of $L^2(\partial \Omega)$ consisting of Steklov eigenfunctions. Then any harmonic function $h\in C^2(\Omega)$ can be represented by
	\begin{equation*}
		h=\sum_{i=0}^{\infty}\left(\int_{\partial\Omega}u_i h\right) u_i.
	\end{equation*}
	Integrating this equation over $\Omega$, we derive
	\begin{equation}\label{eq-4.6}
		\int_{\Omega} h=\left(\int_{\partial\Omega}u_0 h\right) \int_{\Omega} u_0=u_0^2|\Omega|\int_{\partial\Omega}u_0 h,
	\end{equation}
	where we used \eqref{eq-s4} and the fact that $u_0$ is constant.

	On the other hand, since $u_0$ is normalized, we have
	\begin{equation}\label{eq-4.7}
		1=\int_{\partial\Omega}u_0^2=u_0^2|\partial\Omega|.
	\end{equation}
	It follows from \eqref{eq-4.6} and \eqref{eq-4.7} that $h$ satisfies \eqref{eq-4.1}:
	\begin{equation*}
		\frac{1}{|\Omega|}\int_{\Omega} h=\frac{1}{|\partial\Omega|}\int_{\partial\Omega}h.
	\end{equation*}
	By Lemma~\ref{lem-4.1}, $\Omega$ is a geodesic disk.
	\end{proof}

\textbf{Acknowledgment}
	The authors are grateful to Prof.~Changwei Xiong for his helpful suggestions and comments.
	We would also like to thank Prof.~Ailana Fraser and Prof.~Guofang Wei for their interest and useful comments.



\end{document}